\documentclass{amsart}
\usepackage{amsmath}
\usepackage{amsfonts}

\newtheorem{theorem}{Theorem}
\theoremstyle{plain}

\newtheorem{corollary}{Corollary}

\newtheorem{definition}{Definition}

\newtheorem{lemma}{Lemma}

\newtheorem{remark}{Remark}

\numberwithin{equation}{section}
\input{tcilatex}

\begin{document}
\title[On generalization of different type inequalities]{On generalization
of different type inequalities for some convex functions via fractional
integrals}
\author{\.{I}mdat \.{I}\c{s}can}
\address{Department of Mathematics, Faculty of Sciences and Arts, Giresun
University, Giresun, Turkey}
\email{imdat.iscan@giresun.edu.tr}

\begin{abstract}
New identity for fractional integrals have been defined. By using of this
identity, we obtained new estimates on generalization of Hadamard, Ostrowski
and Simpson type inequalities for $s-$convex, quasi-convex, $m-$convex
functions via Riemann Liouville fractional integral.
\end{abstract}

\maketitle

\section{Introduction}

Following inequalities are well known in the literature as Hermite-Hadamard
inequality, Ostrowski inequality and Simpson inequality respectively:

\begin{theorem}
Let $f:I\subseteq \mathbb{R\rightarrow R}$ be a convex function defined on
the interval $I$ of real numbers and $a,b\in I$ with $a<b$. The following
double inequality holds%
\begin{equation}
f\left( \frac{a+b}{2}\right) \leq \frac{1}{b-a}\dint\limits_{a}^{b}f(x)dx%
\leq \frac{f(a)+f(b)}{2}\text{.}  \label{1-1}
\end{equation}
\end{theorem}

\begin{theorem}
Let $f:I\subseteq \mathbb{R\rightarrow R}$ be a mapping differentiable in $%
I^{\circ },$ the interior of I, and let $a,b\in I^{\circ }$ with $a<b.$ If $%
\left\vert f^{\prime }(x)\right\vert \leq M,$ $x\in \left[ a,b\right] ,$
then we the following inequality holds%
\begin{equation*}
\left\vert f(x)-\frac{1}{b-a}\dint\limits_{a}^{b}f(t)dt\right\vert \leq 
\frac{M}{b-a}\left[ \frac{\left( x-a\right) ^{2}+\left( b-x\right) ^{2}}{2}%
\right]
\end{equation*}%
for all $x\in \left[ a,b\right] .$ The constant $\frac{1}{4}$ is the best
possible in the sense that it cannot be replaced by a smaller one.
\end{theorem}

\begin{theorem}
Let $f:\left[ a,b\right] \mathbb{\rightarrow R}$ be a four times
continuously differentiable mapping on $\left( a,b\right) $ and $\left\Vert
f^{(4)}\right\Vert _{\infty }=\underset{x\in \left( a,b\right) }{\sup }%
\left\vert f^{(4)}(x)\right\vert <\infty .$ Then the following inequality
holds:%
\begin{equation*}
\left\vert \frac{1}{3}\left[ \frac{f(a)+f(b)}{2}+2f\left( \frac{a+b}{2}%
\right) \right] -\frac{1}{b-a}\dint\limits_{a}^{b}f(x)dx\right\vert \leq 
\frac{1}{2880}\left\Vert f^{(4)}\right\Vert _{\infty }\left( b-a\right) ^{2}.
\end{equation*}
\end{theorem}

In \cite{HM94}, Hudzik and Maligranda considered among others the class of
functions which are s-convex in the second sense.

\begin{definition}
A function $f:\left[ 0,\infty \right) \rightarrow 
\mathbb{R}
$ is said to be $s-$convex in the second sense if%
\begin{equation*}
f(\alpha x+\beta y)\leq \alpha ^{s}f(x)+\beta ^{s}f(y)
\end{equation*}%
for all $x,y\in \lbrack 0,\infty )$, $\alpha ,\beta \geq 0$ with $\alpha
+\beta =1$ and for some fixed $s\in (0,1]$. This class of $s-$convex
functions in the second sense is usually denoted by $K_{s}^{2}$.
\end{definition}

It can be easily seen that for s = 1, s-convexity reduces to ordinary
convexity of functions defined on $[0,\infty )$.

In \cite{T85}, G. Toader considered the class of m-convexfunctions: another
intermediate between the usual convexity and starshaped convexity.

\begin{definition}
The function $f:[0,b]\rightarrow 
\mathbb{R}
,b>0,$ is said to be $m-$convex, where $m\in \lbrack 0,1]$, if we have%
\begin{equation*}
f(tx+m(1-t)y)\leq tf(x)+m(1-t)f(y)
\end{equation*}%
for all $x,y\in \lbrack 0,b]$ and $t\in \lbrack 0,1]$. We say that $f$ is $%
m- $concave if $(-f)$ is $m-$convex.
\end{definition}

In \cite{PPT92}, Pecaric et al. defined quasi-convex functions as following

\begin{definition}
A function $f:[a,b]\mathbb{\rightarrow R}$ is said quasi-convex on $[a,b]$
if 
\begin{equation*}
f\left( \lambda x+(1-\lambda )y\right) \leq \sup \left\{ f(x),f(y)\right\} ,
\end{equation*}%
for any $x,y\in \lbrack a,b]$ and $\lambda \in \left[ 0,1\right] .$
\end{definition}

The notion of quasi-convex functions generalizes the notion of convex
functions. Clearly, any convex function is a quasi-convex function.
Furthermore, there exist quasi-convex functions which are not \ convex (see 
\cite{I07}).

We give some necessary definitions and mathematical preliminaries of
fractional calculus theory which are used throughout this paper.

\begin{definition}
Let $f\in L\left[ a,b\right] $. The Riemann-Liouville integrals $%
J_{a^{+}}^{\alpha }f$ and $J_{b^{-}}^{\alpha }f$ of oder $\alpha >0$ with $%
a\geq 0$ are defined by

\begin{equation*}
J_{a^{+}}^{\alpha }f(x)=\frac{1}{\Gamma (\alpha )}\dint\limits_{a}^{x}\left(
x-t\right) ^{\alpha -1}f(t)dt,\ x>a
\end{equation*}

and

\begin{equation*}
J_{b^{-}}^{\alpha }f(x)=\frac{1}{\Gamma (\alpha )}\dint\limits_{x}^{b}\left(
t-x\right) ^{\alpha -1}f(t)dt,\ x<b
\end{equation*}%
respectively, where $\Gamma (\alpha )$ is the Gamma function defined by $%
\Gamma (\alpha )=$ $\dint\limits_{0}^{\infty }e^{-t}t^{\alpha -1}dt$ and $%
J_{a^{+}}^{0}f(x)=J_{b^{-}}^{0}f(x)=f(x).$
\end{definition}

In the case of $\alpha =1$, the fractional integral reduces to the classical
integral. Properties concerning this operator can be found \cite%
{GM97,MR93,P99}.

For some recent result connected with fractional integral see \cite%
{D10,OAK12,S12,SSYB11}.

\ 

\section{Generalized integral inequalities for some convex functions via
fractional integrals}

Let $f:I\subseteq 
\mathbb{R}
\rightarrow 
\mathbb{R}
$ be a differentiable function on $I^{\circ }$, the interior of $I$,
throughout this section we will take%
\begin{eqnarray*}
S_{f}\left( x,\lambda ,\alpha ,a,b\right) &=&\left( 1-\lambda \right) \left[ 
\frac{\left( x-a\right) ^{\alpha }+\left( b-x\right) ^{\alpha }}{b-a}\right]
f(x)+\lambda \left[ \frac{\left( x-a\right) ^{\alpha }f(a)+\left( b-x\right)
^{\alpha }f(b)}{b-a}\right] \\
&&-\frac{\Gamma \left( \alpha +1\right) }{b-a}\left[ J_{x^{-}}^{\alpha
}f(a)+J_{x^{+}}^{\alpha }f(b)\right]
\end{eqnarray*}
where $a,b\in I$ with $a<b$, $\ x\in \lbrack a,b]$ , $\lambda \in \left[ 0,1%
\right] $, $\alpha >0$ and $\Gamma $ is Euler Gamma function. In order to
prove our main results we need the following identity.

\begin{lemma}
\label{2.1}Let $f:I\subseteq 
\mathbb{R}
\rightarrow 
\mathbb{R}
$ be a differentiable function on $I^{\circ }$ such that $f^{\prime }\in
L[a,b]$, where $a,b\in I$ with $a<b$. Then for all $x\in \lbrack a,b]$ , $%
\lambda \in \left[ 0,1\right] $ and $\alpha >0$ we have:%
\begin{eqnarray}
&&S_{f}\left( x,\lambda ,\alpha ,a,b\right) =\frac{\left( x-a\right)
^{\alpha +1}}{b-a}\dint\limits_{0}^{1}\left( t^{\alpha }-\lambda \right)
f^{\prime }\left( tx+\left( 1-t\right) a\right) dt  \label{2-1} \\
&&+\frac{\left( b-x\right) ^{\alpha +1}}{b-a}\dint\limits_{0}^{1}\left(
\lambda -t^{\alpha }\right) f^{\prime }\left( tx+\left( 1-t\right) b\right)
dt.  \notag
\end{eqnarray}
\end{lemma}

\begin{proof}
By integration by parts and changing the variable, we can state%
\begin{eqnarray}
&&\dint\limits_{0}^{1}\left( t^{\alpha }-\lambda \right) f^{\prime }\left(
tx+\left( 1-t\right) a\right) dt  \label{2-1a} \\
&=&\left. \left( t^{\alpha }-\lambda \right) \frac{f\left( tx+\left(
1-t\right) a\right) }{x-a}\right\vert _{0}^{1}-\dint\limits_{0}^{1}\alpha
t^{\alpha -1}\frac{f\left( tx+\left( 1-t\right) a\right) }{x-a}dt  \notag \\
&=&\left( 1-\lambda \right) \frac{f(x)}{x-a}+\lambda \frac{f(a)}{x-a}-\frac{%
\alpha }{x-a}\dint\limits_{a}^{x}\left( \frac{u-a}{x-a}\right) ^{\alpha -1}%
\frac{f(u)}{x-a}du  \notag \\
&=&\left( 1-\lambda \right) \frac{f(x)}{x-a}+\lambda \frac{f(a)}{x-a}-\frac{%
\Gamma \left( \alpha +1\right) }{\left( x-a\right) ^{\alpha +1}}%
J_{x^{-}}^{\alpha }f(a)  \notag
\end{eqnarray}%
and similarly we get%
\begin{eqnarray}
&&\dint\limits_{0}^{1}\left( \lambda -t^{\alpha }\right) f^{\prime }\left(
tx+\left( 1-t\right) b\right) dt  \label{2-1b} \\
&=&\left. \left( \lambda -t^{\alpha }\right) \frac{f\left( tx+\left(
1-t\right) b\right) }{x-b}\right\vert _{0}^{1}-\dint\limits_{0}^{1}\alpha
t^{\alpha -1}\frac{f\left( tx+\left( 1-t\right) b\right) }{x-b}dt  \notag \\
&=&\left( 1-\lambda \right) \frac{f(x)}{b-x}+\lambda \frac{f(b)}{b-x}-\frac{%
\alpha }{b-x}\dint\limits_{x}^{b}\left( \frac{b-u}{b-x}\right) ^{\alpha -1}%
\frac{f(u)}{b-x}du  \notag \\
&=&\left( 1-\lambda \right) \frac{f(x)}{b-x}+\lambda \frac{f(b)}{b-x}-\frac{%
\Gamma \left( \alpha +1\right) }{\left( b-x\right) ^{\alpha +1}}%
J_{x^{+}}^{\alpha }f(b)  \notag
\end{eqnarray}%
Multiplying both sides of (\ref{2-1a}) and (\ref{2-1b}) by $\frac{\left(
x-a\right) ^{\alpha +1}}{b-a}$ and $\frac{\left( b-x\right) ^{\alpha +1}}{b-a%
}$, respectively, and adding the resulting identities we obtain the desired
result.
\end{proof}

\subsection{For $s-$convex functions.}

\begin{theorem}
\label{2.1.1}Let $f:$ $I\subset \lbrack 0,\infty )\rightarrow 
\mathbb{R}
$ be a differentiable function on $I^{\circ }$ such that $f^{\prime }\in
L[a,b]$, where $a,b\in I^{\circ }$ with $a<b$. If $|f^{\prime }|^{q}$ is $s-$%
convex on $[a,b]$ for some fixed $q\geq 1$, $x\in \lbrack a,b]$, $\lambda
\in \left[ 0,1\right] $ then the following inequality for fractional
integrals holds%
\begin{eqnarray}
&&\left\vert S_{f}\left( x,\lambda ,\alpha ,a,b\right) \right\vert   \notag
\\
&\leq &A_{1}^{1-\frac{1}{q}}\left( \alpha ,\lambda \right) \left\{ \frac{%
\left( x-a\right) ^{\alpha +1}}{b-a}\left( \left\vert f^{\prime }\left(
x\right) \right\vert ^{q}A_{2}\left( \alpha ,\lambda \right) +\left\vert
f^{\prime }\left( a\right) \right\vert ^{q}A_{3}\left( \alpha ,\lambda
\right) \right) ^{\frac{1}{q}}\right.   \label{2-2} \\
&&+\left. \frac{\left( b-x\right) ^{\alpha +1}}{b-a}\left( \left\vert
f^{\prime }\left( x\right) \right\vert ^{q}A_{2}\left( \alpha ,\lambda
\right) +\left\vert f^{\prime }\left( b\right) \right\vert ^{q}A_{3}\left(
\alpha ,\lambda \right) \right) ^{\frac{1}{q}}\right\}   \notag
\end{eqnarray}%
where 
\begin{eqnarray*}
A_{1}\left( \alpha ,\lambda \right)  &=&\frac{2\alpha \lambda ^{1+\frac{1}{%
\alpha }}+1}{\alpha +1}-\lambda , \\
A_{2}\left( \alpha ,\lambda ,s\right)  &=&\frac{2\alpha \lambda ^{1+\frac{s+1%
}{\alpha }}+s+1}{\left( s+1\right) \left( \alpha +s+1\right) }-\frac{\lambda 
}{s+1}, \\
A_{3}\left( \alpha ,\lambda ,s\right)  &=&\lambda \left[ \frac{1-2\left(
1-\lambda ^{\frac{1}{\alpha }}\right) ^{s+1}}{s+1}\right] +\beta \left(
\alpha +1,s+1\right) -2\beta \left( \lambda ^{\frac{1}{\alpha }};\alpha
+1,s+1\right) ,
\end{eqnarray*}%
and $\beta $ is Euler Beta function defined by%
\begin{equation*}
\beta \left( x,y\right) =\frac{\Gamma (x)\Gamma (y)}{\Gamma (x+y)}%
=\dint\limits_{0}^{1}t^{x-1}\left( 1-t\right) ^{y-1}dt,\ \ x,y>0.
\end{equation*}
\end{theorem}

\begin{proof}
From Lemma \ref{2.1}, property of the modulus and using the power-mean
inequality we have%
\begin{eqnarray}
&&\left\vert S_{f}\left( x,\lambda ,\alpha ,a,b\right) \right\vert \leq 
\frac{\left( x-a\right) ^{\alpha +1}}{b-a}\dint\limits_{0}^{1}\left\vert
t^{\alpha }-\lambda \right\vert \left\vert f^{\prime }\left( tx+\left(
1-t\right) a\right) \right\vert dt  \notag \\
&&+\frac{\left( b-x\right) ^{\alpha +1}}{b-a}\dint\limits_{0}^{1}\left\vert
\lambda -t^{\alpha }\right\vert \left\vert f^{\prime }\left( tx+\left(
1-t\right) b\right) \right\vert dt  \notag \\
&\leq &\frac{\left( x-a\right) ^{\alpha +1}}{b-a}\left(
\dint\limits_{0}^{1}\left\vert t^{\alpha }-\lambda \right\vert dt\right) ^{1-%
\frac{1}{q}}\left( \dint\limits_{0}^{1}\left\vert t^{\alpha }-\lambda
\right\vert \left\vert f^{\prime }\left( tx+\left( 1-t\right) a\right)
\right\vert ^{q}dt\right) ^{\frac{1}{q}}  \notag \\
&&+\frac{\left( b-x\right) ^{\alpha +1}}{b-a}\left(
\dint\limits_{0}^{1}\left\vert t^{\alpha }-\lambda \right\vert dt\right) ^{1-%
\frac{1}{q}}\left( \dint\limits_{0}^{1}\left\vert t^{\alpha }-\lambda
\right\vert \left\vert f^{\prime }\left( tx+\left( 1-t\right) b\right)
\right\vert ^{q}dt\right) ^{\frac{1}{q}}  \label{2-2a}
\end{eqnarray}%
Since$\left\vert f^{\prime }\right\vert ^{q}$ is $s-$convex on $[a,b]$ we get%
\begin{eqnarray}
\dint\limits_{0}^{1}\left\vert t^{\alpha }-\lambda \right\vert \left\vert
f^{\prime }\left( tx+\left( 1-t\right) a\right) \right\vert ^{q}dt &\leq
&\dint\limits_{0}^{1}\left\vert t^{\alpha }-\lambda \right\vert \left(
t^{s}\left\vert f^{\prime }\left( x\right) \right\vert ^{q}+\left(
1-t\right) ^{s}\left\vert f^{\prime }\left( a\right) \right\vert ^{q}\right)
dt  \notag \\
&=&\left\vert f^{\prime }\left( x\right) \right\vert ^{q}A_{2}\left( \alpha
,\lambda ,s\right) +\left\vert f^{\prime }\left( a\right) \right\vert
^{q}A_{3}\left( \alpha ,\lambda ,s\right) ,  \label{2-2b}
\end{eqnarray}%
\begin{eqnarray}
\dint\limits_{0}^{1}\left\vert t^{\alpha }-\lambda \right\vert \left\vert
f^{\prime }\left( tx+\left( 1-t\right) b\right) \right\vert ^{q}dt &\leq
&\dint\limits_{0}^{1}\left\vert t^{\alpha }-\lambda \right\vert \left(
t^{s}\left\vert f^{\prime }\left( x\right) \right\vert ^{q}+\left(
1-t\right) ^{s}\left\vert f^{\prime }\left( b\right) \right\vert ^{q}\right)
dt  \notag \\
&=&\left\vert f^{\prime }\left( x\right) \right\vert ^{q}A_{2}\left( \alpha
,\lambda ,s\right) +\left\vert f^{\prime }\left( b\right) \right\vert
^{q}A_{3}\left( \alpha ,\lambda ,s\right) ,  \label{2-2c}
\end{eqnarray}%
where we use the fact that%
\begin{eqnarray*}
\dint\limits_{0}^{1}\left\vert t^{\alpha }-\lambda \right\vert \left(
1-t\right) ^{s}dt &=&\dint\limits_{0}^{\lambda ^{\frac{1}{\alpha }}}\left(
\lambda -t^{\alpha }\right) \left( 1-t\right) ^{s}dt+\dint\limits_{\lambda ^{%
\frac{1}{\alpha }}}^{1}\left( t^{\alpha }-\lambda \right) \left( 1-t\right)
^{s}dt \\
&=&\lambda \dint\limits_{0}^{\lambda ^{\frac{1}{\alpha }}}\left( 1-t\right)
^{s}dt-\dint\limits_{0}^{\lambda ^{\frac{1}{\alpha }}}t^{\alpha }\left(
1-t\right) ^{s}dt+\dint\limits_{\lambda ^{\frac{1}{\alpha }}}^{1}t^{\alpha
}\left( 1-t\right) ^{s}dt \\
&&-\lambda \dint\limits_{\lambda ^{\frac{1}{\alpha }}}^{1}\left( 1-t\right)
^{s}dt \\
&=&\lambda \left[ \frac{1-2\left( 1-\lambda ^{\frac{1}{\alpha }}\right)
^{s+1}}{s+1}\right] +\dint\limits_{0}^{1}t^{\alpha }\left( 1-t\right)
^{s}dt-2\dint\limits_{0}^{\lambda ^{\frac{1}{\alpha }}}t^{\alpha }\left(
1-t\right) ^{s}dt \\
&=&\lambda \left[ \frac{1-2\left( 1-\lambda ^{\frac{1}{\alpha }}\right)
^{s+1}}{s+1}\right] +\beta \left( \alpha +1,s+1\right) -2\beta \left(
\lambda ^{\frac{1}{\alpha }};\alpha +1,s+1\right) ,
\end{eqnarray*}%
\begin{equation*}
\dint\limits_{0}^{1}\left\vert t^{\alpha }-\lambda \right\vert t^{s}dt=\frac{%
2\alpha \lambda ^{1+\frac{s+1}{\alpha }}+s+1}{\left( s+1\right) \left(
\alpha +s+1\right) }-\frac{\lambda }{s+1},
\end{equation*}%
and by simple computation%
\begin{eqnarray}
\dint\limits_{0}^{1}\left\vert t^{\alpha }-\lambda \right\vert dt
&=&\dint\limits_{0}^{\lambda ^{\frac{1}{\alpha }}}\left( \lambda -t^{\alpha
}\right) dt+\dint\limits_{\lambda ^{\frac{1}{\alpha }}}^{1}\left( t^{\alpha
}-\lambda \right) dt  \notag \\
&=&\frac{2\alpha \lambda ^{1+\frac{1}{\alpha }}+1}{\alpha +1}-\lambda .
\label{2-2d}
\end{eqnarray}%
Hence, If we use (\ref{2-2b}), (\ref{2-2c}) and (\ref{2-2d}) in (\ref{2-2a}%
), we obtain the desired result. This completes the proof.
\end{proof}

\begin{corollary}
Under the assumptions of Theorem \ref{2.1.1} with $q=1,$ the inequality (\ref%
{2-2}) reduced to the following inequality%
\begin{eqnarray*}
\left\vert S_{f}\left( x,\lambda ,\alpha ,a,b\right) \right\vert  &\leq
&\left\{ \frac{\left( x-a\right) ^{\alpha +1}}{b-a}\left( \left\vert
f^{\prime }\left( x\right) \right\vert A_{2}\left( \alpha ,\lambda ,s\right)
+\left\vert f^{\prime }\left( a\right) \right\vert A_{3}\left( \alpha
,\lambda ,s\right) \right) \right.  \\
&&+\left. \frac{\left( b-x\right) ^{\alpha +1}}{b-a}\left( \left\vert
f^{\prime }\left( x\right) \right\vert A_{2}\left( \alpha ,\lambda ,s\right)
+\left\vert f^{\prime }\left( b\right) \right\vert A_{3}\left( \alpha
,\lambda ,s\right) \right) \right\} .
\end{eqnarray*}
\end{corollary}

\begin{corollary}
Under the assumptions of Theorem \ref{2.1.1} with $x=\frac{a+b}{2},\ \lambda
=\frac{1}{3},$from the inequality (\ref{2-2}) we get the following Simpson
type inequality for fractional integrals%
\begin{eqnarray*}
&&\left\vert \frac{2^{\alpha -1}}{\left( b-a\right) ^{\alpha -1}}S_{f}\left( 
\frac{a+b}{2},\frac{1}{3},\alpha ,a,b\right) \right\vert  \\
&=&\left\vert \frac{1}{6}\left[ f(a)+4f\left( \frac{a+b}{2}\right) +f(b)%
\right] -\frac{\Gamma \left( \alpha +1\right) 2^{\alpha -1}}{\left(
b-a\right) ^{\alpha }}\left[ J_{\left( \frac{a+b}{2}\right) ^{-}}^{\alpha
}f(a)+J_{\left( \frac{a+b}{2}\right) ^{+}}^{\alpha }f(b)\right] \right\vert 
\\
&\leq &\frac{b-a}{4}A_{1}^{1-\frac{1}{q}}\left( \alpha ,\frac{1}{3}\right)
\left\{ \left( \left\vert f^{\prime }\left( \frac{a+b}{2}\right) \right\vert
^{q}A_{2}\left( \alpha ,\frac{1}{3},s\right) +\left\vert f^{\prime }\left(
a\right) \right\vert ^{q}A_{3}\left( \alpha ,\frac{1}{3},s\right) \right) ^{%
\frac{1}{q}}\right.  \\
&&\left. +\left( \left\vert f^{\prime }\left( \frac{a+b}{2}\right)
\right\vert ^{q}A_{2}\left( \alpha ,\frac{1}{3},s\right) +\left\vert
f^{\prime }\left( b\right) \right\vert ^{q}A_{3}\left( \alpha ,\frac{1}{3}%
,s\right) \right) ^{\frac{1}{q}}\right\} .
\end{eqnarray*}
\end{corollary}

\begin{corollary}
Under the assumptions of Theorem \ref{2.1.1} with $x=\frac{a+b}{2},\ \lambda
=0,$from the inequality (\ref{2-2}) we get the following midpoint type
inequality for fractional integrals%
\begin{eqnarray*}
&&\left\vert \frac{2^{\alpha -1}}{\left( b-a\right) ^{\alpha -1}}S_{f}\left( 
\frac{a+b}{2},0,\alpha ,a,b\right) \right\vert  \\
&=&\left\vert f\left( \frac{a+b}{2}\right) -\frac{\Gamma \left( \alpha
+1\right) 2^{\alpha -1}}{\left( b-a\right) ^{\alpha }}\left[ J_{\left( \frac{%
a+b}{2}\right) ^{-}}^{\alpha }f(a)+J_{\left( \frac{a+b}{2}\right)
^{+}}^{\alpha }f(b)\right] \right\vert  \\
&\leq &\frac{b-a}{4}\left( \frac{1}{\alpha +1}\right) ^{1-\frac{1}{q}%
}\left\{ \left[ \frac{\left\vert f^{\prime }\left( \frac{a+b}{2}\right)
\right\vert ^{q}}{\alpha +s+1}+\left\vert f^{\prime }\left( a\right)
\right\vert ^{q}\beta \left( \alpha +1,s+1\right) \right] ^{\frac{1}{q}%
}\right.  \\
&&\left. +\left[ \frac{\left\vert f^{\prime }\left( \frac{a+b}{2}\right)
\right\vert ^{q}}{\alpha +s+1}+\left\vert f^{\prime }\left( b\right)
\right\vert ^{q}\beta \left( \alpha +1,s+1\right) \right] ^{\frac{1}{q}%
}\right\} .
\end{eqnarray*}
\end{corollary}

\begin{corollary}
\label{2.1.1a}Under the assumptions of Theorem \ref{2.1.1} with$\ \lambda =1,
$from the inequality (\ref{2-2}) we get the following trapezoid type
inequality for fractional integrals%
\begin{eqnarray*}
&&\left\vert S_{f}\left( \frac{a+b}{2},1,\alpha ,a,b\right) \right\vert  \\
&=&\left\vert \frac{\left( x-a\right) ^{\alpha }f(a)+\left( b-x\right)
^{\alpha }f(b)}{b-a}-\frac{\Gamma \left( \alpha +1\right) }{b-a}\left[
J_{x^{-}}^{\alpha }f(a)+J_{x^{+}}^{\alpha }f(b)\right] \right\vert  \\
&\leq &\left( \frac{\alpha }{\alpha +1}\right) ^{1-\frac{1}{q}}\left\{ \frac{%
\left( x-a\right) ^{\alpha +1}}{b-a}\left[ \frac{\alpha \left\vert f^{\prime
}\left( x\right) \right\vert ^{q}}{\left( s+1\right) \left( \alpha
+s+1\right) }+\left\vert f^{\prime }\left( a\right) \right\vert ^{q}\left( 
\frac{1}{s+1}-\beta \left( \alpha +1,s+1\right) \right) \right] ^{\frac{1}{q}%
}\right.  \\
&&\left. +\frac{\left( b-x\right) ^{\alpha +1}}{b-a}\left[ \frac{\alpha
\left\vert f^{\prime }\left( x\right) \right\vert ^{q}}{\left( s+1\right)
\left( \alpha +s+1\right) }+\left\vert f^{\prime }\left( b\right)
\right\vert ^{q}\left( \frac{1}{s+1}-\beta \left( \alpha +1,s+1\right)
\right) \right] ^{\frac{1}{q}}\right\} 
\end{eqnarray*}%
which is the same of the inequality in \cite[Theorem 9]{OAK12}.
\end{corollary}

\begin{remark}
In Corollary \ref{2.1.1a}, if we choose $\alpha =1,$ we get the same
inequality in \cite[Theorem 7]{AKO11}.
\end{remark}

\begin{corollary}
\label{2.1.1b}Let the assumptions of Theorem \ref{2.1.1} hold. If $\
\left\vert f^{\prime }(x)\right\vert \leq M$ for all $x\in \left[ a,b\right] 
$ and $\lambda =0,$ then from the inequality (\ref{2-2}) we get the
following Ostrowski type inequality for fractional integrals%
\begin{eqnarray*}
&&\left\vert \left[ \frac{\left( x-a\right) ^{\alpha }+\left( b-x\right)
^{\alpha }}{b-a}\right] f(x)-\frac{\Gamma \left( \alpha +1\right) }{b-a}%
\left[ J_{x^{-}}^{\alpha }f(a)+J_{x^{+}}^{\alpha }f(b)\right] \right\vert  \\
&\leq &M\left( \frac{1}{\alpha +1}\right) ^{1-\frac{1}{q}}\left[ \frac{1}{%
\alpha +s+1}+\beta \left( \alpha +1,s+1\right) \right] ^{\frac{1}{q}}\left[ 
\frac{\left( x-a\right) ^{\alpha +1}+\left( b-x\right) ^{\alpha +1}}{b-a}%
\right] 
\end{eqnarray*}%
for each $x\in \left[ a,b\right] .$
\end{corollary}

\begin{remark}
In Corollary \ref{2.1.1b}, if we choose $\alpha =1,$ we get the same
inequality in \cite[Theorem 4]{ADDC10}.
\end{remark}

\subsection{For quasi-convex functions}

\begin{theorem}
\label{2.2}Let $f:$ $I\subset \lbrack 0,\infty )\rightarrow 
\mathbb{R}
$ be a differentiable function on $I^{\circ }$ such that $f^{\prime }\in
L[a,b]$, where $a,b\in I^{\circ }$ with $a<b$. If $|f^{\prime }|^{q}$ is
quasi-convex on $[a,b]$ for some fixed $q\geq 1$, $x\in \lbrack a,b]$, $%
\lambda \in \left[ 0,1\right] $ then the following inequality for fractional
integrals holds%
\begin{eqnarray}
\left\vert S_{f}\left( x,\lambda ,\alpha ,a,b\right) \right\vert  &\leq
&A_{1}\left( \alpha ,\lambda \right) \left\{ \frac{\left( x-a\right)
^{\alpha +1}}{b-a}\left( \sup \left\{ \left\vert f^{\prime }\left( x\right)
\right\vert ^{q},\left\vert f^{\prime }\left( a\right) \right\vert
^{q}\right\} \right) ^{\frac{1}{q}}\right.   \label{2-3} \\
&&\left. +\frac{\left( b-x\right) ^{\alpha +1}}{b-a}\left( \sup \left\{
\left\vert f^{\prime }\left( x\right) \right\vert ^{q},\left\vert f^{\prime
}\left( b\right) \right\vert ^{q}\right\} \right) ^{\frac{1}{q}}\right\}  
\notag
\end{eqnarray}%
where 
\begin{equation*}
A_{1}\left( \alpha ,\lambda \right) =\frac{2\alpha \lambda ^{1+\frac{1}{%
\alpha }}+1}{\alpha +1}-\lambda .
\end{equation*}
\end{theorem}

\begin{proof}
We proceed similarly as in the proof Theorem \ref{2.1.1}. Since $\left\vert
f^{\prime }\right\vert ^{q}$ is quasi-convex on $[a,b],$ for all $t\in \left[
0,1\right] $%
\begin{equation*}
\left\vert f^{\prime }\left( tx+\left( 1-t\right) a\right) \right\vert
^{q}\leq \sup \left\{ \left\vert f^{\prime }\left( x\right) \right\vert
^{q},\left\vert f^{\prime }\left( a\right) \right\vert ^{q}\right\} 
\end{equation*}%
and%
\begin{equation*}
\left\vert f^{\prime }\left( tx+\left( 1-t\right) b\right) \right\vert
^{q}\leq \sup \left\{ \left\vert f^{\prime }\left( x\right) \right\vert
^{q},\left\vert f^{\prime }\left( b\right) \right\vert ^{q}\right\} .
\end{equation*}%
Hence, from the inequality (\ref{2-2a}) we get%
\begin{eqnarray*}
&&\left\vert S_{f}\left( x,\lambda ,\alpha ,a,b\right) \right\vert  \\
&\leq &\left( \dint\limits_{0}^{1}\left\vert t^{\alpha }-\lambda \right\vert
dt\right) \left\{ \frac{\left( x-a\right) ^{\alpha +1}}{b-a}\left( \sup
\left\{ \left\vert f^{\prime }\left( x\right) \right\vert ^{q},\left\vert
f^{\prime }\left( a\right) \right\vert ^{q}\right\} \right) ^{\frac{1}{q}%
}\right.  \\
&&\left. +\frac{\left( b-x\right) ^{\alpha +1}}{b-a}\left( \sup \left\{
\left\vert f^{\prime }\left( x\right) \right\vert ^{q},\left\vert f^{\prime
}\left( b\right) \right\vert ^{q}\right\} \right) ^{\frac{1}{q}}\right\}  \\
&\leq &A_{1}\left( \alpha ,\lambda \right) \left\{ \frac{\left( x-a\right)
^{\alpha +1}}{b-a}\left( \sup \left\{ \left\vert f^{\prime }\left( x\right)
\right\vert ^{q},\left\vert f^{\prime }\left( a\right) \right\vert
^{q}\right\} \right) ^{\frac{1}{q}}\right.  \\
&&\left. +\frac{\left( b-x\right) ^{\alpha +1}}{b-a}\left( \sup \left\{
\left\vert f^{\prime }\left( x\right) \right\vert ^{q},\left\vert f^{\prime
}\left( b\right) \right\vert ^{q}\right\} \right) ^{\frac{1}{q}}\right\} 
\end{eqnarray*}%
which completes the proof.
\end{proof}

\begin{corollary}
Under the assumptions of Theorem \ref{2.2} with $q=1,$ the inequality (\ref%
{2-3}) reduced to the following inequality%
\begin{eqnarray*}
\left\vert S_{f}\left( x,\lambda ,\alpha ,a,b\right) \right\vert  &\leq
&A_{1}\left( \alpha ,\lambda \right) \left\{ \frac{\left( x-a\right)
^{\alpha +1}}{b-a}\left( \sup \left\{ \left\vert f^{\prime }\left( x\right)
\right\vert ,\left\vert f^{\prime }\left( a\right) \right\vert \right\}
\right) \right.  \\
&&+\left. \frac{\left( b-x\right) ^{\alpha +1}}{b-a}\left( \sup \left\{
\left\vert f^{\prime }\left( x\right) \right\vert ,\left\vert f^{\prime
}\left( b\right) \right\vert \right\} \right) \right\} .
\end{eqnarray*}
\end{corollary}

\begin{corollary}
Under the assumptions of Theorem \ref{2.2} with $x=\frac{a+b}{2},\ \lambda =%
\frac{1}{3},$from the inequality (\ref{2-3}) we get the following Simpson
type inequality or fractional integrals%
\begin{eqnarray*}
&&\left\vert \frac{1}{6}\left[ f(a)+4f\left( \frac{a+b}{2}\right) +f(b)%
\right] -\frac{\Gamma \left( \alpha +1\right) 2^{\alpha -1}}{\left(
b-a\right) ^{\alpha }}\left[ J_{\left( \frac{a+b}{2}\right) ^{-}}^{\alpha
}f(a)+J_{\left( \frac{a+b}{2}\right) ^{+}}^{\alpha }f(b)\right] \right\vert 
\\
&\leq &\frac{b-a}{4}A_{1}\left( \alpha ,\frac{1}{3}\right) \left[ \sup
\left\{ \left\vert f^{\prime }\left( \frac{a+b}{2}\right) \right\vert
,\left\vert f^{\prime }\left( a\right) \right\vert \right\} +\sup \left\{
\left\vert f^{\prime }\left( \frac{a+b}{2}\right) \right\vert ,\left\vert
f^{\prime }\left( b\right) \right\vert \right\} \right] .
\end{eqnarray*}
\end{corollary}

\begin{corollary}
Under the assumptions of Theorem \ref{2.2} with $x=\frac{a+b}{2},\ \lambda
=0,$from the inequality (\ref{2-3}) we get the following midpoint type
inequality or fractional integrals%
\begin{eqnarray*}
&&\left\vert f\left( \frac{a+b}{2}\right) -\frac{\Gamma \left( \alpha
+1\right) 2^{\alpha -1}}{\left( b-a\right) ^{\alpha }}\left[ J_{\left( \frac{%
a+b}{2}\right) ^{-}}^{\alpha }f(a)+J_{\left( \frac{a+b}{2}\right)
^{+}}^{\alpha }f(b)\right] \right\vert  \\
&\leq &\frac{b-a}{4}\left( \frac{1}{\alpha +1}\right) \left\{ \left[ \sup
\left\{ \left\vert f^{\prime }\left( \frac{a+b}{2}\right) \right\vert
^{q},\left\vert f^{\prime }\left( a\right) \right\vert ^{q}\right\} \right]
^{\frac{1}{q}}\right.  \\
&&\left. +\left[ \sup \left\{ \left\vert f^{\prime }\left( \frac{a+b}{2}%
\right) \right\vert ^{q},\left\vert f^{\prime }\left( b\right) \right\vert
^{q}\right\} \right] ^{\frac{1}{q}}\right\} .
\end{eqnarray*}
\end{corollary}

\begin{corollary}
Under the assumptions of Theorem \ref{2.2} with$\ \lambda =1,$from the
inequality (\ref{2-3}) we get the following trapezoid type inequality or
fractional integrals%
\begin{eqnarray*}
&&\left\vert \frac{\left( x-a\right) ^{\alpha }f(a)+\left( b-x\right)
^{\alpha }f(b)}{b-a}-\frac{\Gamma \left( \alpha +1\right) }{b-a}\left[
J_{x^{-}}^{\alpha }f(a)+J_{x^{+}}^{\alpha }f(b)\right] \right\vert  \\
&\leq &\left( \frac{\alpha }{\alpha +1}\right) \left\{ \frac{\left(
x-a\right) ^{\alpha +1}}{b-a}\left[ \sup \left\{ \left\vert f^{\prime
}\left( x\right) \right\vert ^{q},\left\vert f^{\prime }\left( a\right)
\right\vert ^{q}\right\} \right] ^{\frac{1}{q}}\right.  \\
&&\left. +\frac{\left( b-x\right) ^{\alpha +1}}{b-a}\left[ \sup \left\{
\left\vert f^{\prime }\left( x\right) \right\vert ^{q},\left\vert f^{\prime
}\left( b\right) \right\vert ^{q}\right\} \right] ^{\frac{1}{q}}\right\} .
\end{eqnarray*}
\end{corollary}

\begin{corollary}
Let the assumptions of Theorem \ref{2.2} hold. If $\ \left\vert f^{\prime
}(x)\right\vert \leq M$ for all $x\in \left[ a,b\right] $ and $\lambda =0,$
then from the inequality (\ref{2-3}) we get the following Ostrowski type
inequality or fractional integrals%
\begin{eqnarray*}
&&\left\vert \left[ \frac{\left( x-a\right) ^{\alpha }+\left( b-x\right)
^{\alpha }}{b-a}\right] f(x)-\frac{\Gamma \left( \alpha +1\right) }{b-a}%
\left[ J_{x^{-}}^{\alpha }f(a)+J_{x^{+}}^{\alpha }f(b)\right] \right\vert  \\
&\leq &\frac{M}{\alpha +1}\left[ \frac{\left( x-a\right) ^{\alpha +1}+\left(
b-x\right) ^{\alpha +1}}{b-a}\right] ,
\end{eqnarray*}
\end{corollary}

\subsection{For $m-$convex functions}

Similarly lemma \ref{2.1}, we can proved the following lemma

\begin{lemma}
\label{2.1a},Let $f:I\subset \mathbb{R\rightarrow R}$ be a differentiable
mapping on $I^{\circ }$ such that $f^{\prime }\in L[ma,mb]$, where $m\in
\left( 0,1\right] $, $ma,mb\in I$ with $a<b$. Then for all $x\in \lbrack
a,b] $ , $\lambda \in \left[ 0,1\right] $ and $\alpha >0$ we have:%
\begin{eqnarray*}
&&S_{f}\left( mx,\lambda ,\alpha ,ma,mb\right) =\frac{m^{\alpha }\left(
x-a\right) ^{\alpha +1}}{b-a}\dint\limits_{0}^{1}\left( t^{\alpha }-\lambda
\right) f^{\prime }\left( tmx+m\left( 1-t\right) a\right) dt \\
&&+\frac{m^{\alpha }\left( b-x\right) ^{\alpha +1}}{b-a}\dint\limits_{0}^{1}%
\left( \lambda -t^{\alpha }\right) f^{\prime }\left( tmx+m\left( 1-t\right)
b\right) dt.
\end{eqnarray*}
\end{lemma}

\begin{theorem}
\label{2.3}Let $f:$ $I\subset \lbrack 0,\infty )\rightarrow 
\mathbb{R}
$ be a differentiable function on $I^{\circ }$ such that $f^{\prime }\in
L[ma,mb]$, where $m\in \left( 0,1\right] $, $a,b\in I$ $^{\circ }$ with $a<b$%
. If $|f^{\prime }|^{q}$ is $m-$convex on $[ma,mb]$ for some fixed $q\geq 1$%
, $x\in \lbrack a,b]$, $\lambda \in \left[ 0,1\right] $ and $\alpha >0$ then
the following inequality for fractional integrals holds%
\begin{eqnarray}
&&\left\vert S_{f}\left( mx,\lambda ,\alpha ,ma,mb\right) \right\vert 
\label{2-4} \\
&\leq &A_{1}\left( \alpha ,\lambda \right) ^{1-\frac{1}{q}}\left\{ \frac{%
m^{\alpha }\left( x-a\right) ^{\alpha +1}}{b-a}\left( \left\vert f^{\prime
}\left( mx\right) \right\vert ^{q}A_{2}\left( \alpha ,\lambda \right)
+m\left\vert f^{\prime }\left( a\right) \right\vert ^{q}A_{3}\left( \alpha
,\lambda \right) \right) ^{\frac{1}{q}}\right.   \notag \\
&&\left. +\frac{m^{\alpha }\left( b-x\right) ^{\alpha +1}}{b-a}\left(
\left\vert f^{\prime }\left( mx\right) \right\vert ^{q}A_{2}\left( \alpha
,\lambda \right) +m\left\vert f^{\prime }\left( b\right) \right\vert
^{q}A_{3}\left( \alpha ,\lambda \right) \right) ^{\frac{1}{q}}\right\}  
\notag
\end{eqnarray}%
where 
\begin{eqnarray*}
A_{1}\left( \alpha ,\lambda \right)  &=&\frac{2\alpha \lambda ^{1+\frac{1}{%
\alpha }}+1}{\alpha +1}-\lambda , \\
A_{2}\left( \alpha ,\lambda \right)  &=&\frac{\alpha \lambda ^{1+\frac{2}{%
\alpha }}+1}{\alpha +2}-\frac{\lambda }{2}, \\
A_{3}\left( \alpha ,\lambda \right)  &=&\frac{2\alpha \lambda ^{1+\frac{1}{%
\alpha }}+1}{\alpha +1}-\frac{\alpha \lambda ^{1+\frac{2}{\alpha }}+1}{%
\alpha +2}-\frac{\lambda }{2}.
\end{eqnarray*}
\end{theorem}

\begin{proof}
We proceed similarly as in the proof Theorem \ref{2.1.1}. From Lemma \ref%
{2.1a}, property of the modulus and using the power-mean inequality we have%
\begin{eqnarray}
&&\left\vert S_{f}\left( mx,\lambda ,\alpha ,ma,mb\right) \right\vert  
\notag \\
&\leq &\frac{m^{\alpha }\left( x-a\right) ^{\alpha +1}}{b-a}\left(
\dint\limits_{0}^{1}\left\vert t^{\alpha }-\lambda \right\vert dt\right) ^{1-%
\frac{1}{q}}\left( \dint\limits_{0}^{1}\left\vert t^{\alpha }-\lambda
\right\vert \left\vert f^{\prime }\left( tmx+m\left( 1-t\right) a\right)
\right\vert ^{q}dt\right) ^{\frac{1}{q}}  \notag \\
&&+\frac{m^{\alpha }\left( b-x\right) ^{\alpha +1}}{b-a}\left(
\dint\limits_{0}^{1}\left\vert t^{\alpha }-\lambda \right\vert dt\right) ^{1-%
\frac{1}{q}}\left( \dint\limits_{0}^{1}\left\vert t^{\alpha }-\lambda
\right\vert \left\vert f^{\prime }\left( tmx+m\left( 1-t\right) b\right)
\right\vert ^{q}dt\right) ^{\frac{1}{q}}.  \label{2-4a}
\end{eqnarray}%
Since $|f^{\prime }|^{q}$ is $m-$convex on $[ma,mb],$ for all $t\in \left[
0,1\right] $%
\begin{equation*}
\left\vert f^{\prime }\left( tmx+m\left( 1-t\right) a\right) \right\vert
^{q}\leq t\left\vert f^{\prime }\left( mx\right) \right\vert ^{q}+m\left(
1-t\right) \left\vert f^{\prime }\left( a\right) \right\vert ^{q}
\end{equation*}%
and%
\begin{equation*}
\left\vert f^{\prime }\left( tmx+m\left( 1-t\right) b\right) \right\vert
^{q}\leq t\left\vert f^{\prime }\left( mx\right) \right\vert ^{q}+m\left(
1-t\right) \left\vert f^{\prime }\left( b\right) \right\vert ^{q}.
\end{equation*}%
Hence by simple computation we get%
\begin{eqnarray}
\dint\limits_{0}^{1}\left\vert t^{\alpha }-\lambda \right\vert \left\vert
f^{\prime }\left( tmx+m\left( 1-t\right) a\right) \right\vert ^{q}dt &\leq
&\dint\limits_{0}^{1}\left\vert t^{\alpha }-\lambda \right\vert t\left\vert
f^{\prime }\left( mx\right) \right\vert ^{q}+m\left( 1-t\right) \left\vert
f^{\prime }\left( a\right) \right\vert ^{q}dt  \notag \\
&=&\left\vert f^{\prime }\left( mx\right) \right\vert ^{q}\left( \frac{%
\alpha \lambda ^{1+\frac{2}{\alpha }}+1}{\alpha +2}-\frac{\lambda }{2}%
\right)   \notag \\
&&+m\left\vert f^{\prime }\left( a\right) \right\vert ^{q}\left( \frac{%
2\alpha \lambda ^{1+\frac{1}{\alpha }}+1}{\alpha +1}-\frac{\alpha \lambda
^{1+\frac{2}{\alpha }}+1}{\alpha +2}-\frac{\lambda }{2}\right) ,
\label{2-4b}
\end{eqnarray}%
and similarly%
\begin{eqnarray}
\dint\limits_{0}^{1}\left\vert t^{\alpha }-\lambda \right\vert \left\vert
f^{\prime }\left( tmx+m\left( 1-t\right) b\right) \right\vert ^{q}dt &\leq
&\dint\limits_{0}^{1}\left\vert t^{\alpha }-\lambda \right\vert t\left\vert
f^{\prime }\left( mx\right) \right\vert ^{q}+m\left( 1-t\right) \left\vert
f^{\prime }\left( b\right) \right\vert ^{q}dt  \notag \\
&=&\left\vert f^{\prime }\left( mx\right) \right\vert ^{q}\left( \frac{%
\alpha \lambda ^{1+\frac{2}{\alpha }}+1}{\alpha +2}-\frac{\lambda }{2}%
\right)   \notag \\
&&+m\left\vert f^{\prime }\left( b\right) \right\vert ^{q}\left( \frac{%
2\alpha \lambda ^{1+\frac{1}{\alpha }}+1}{\alpha +1}-\frac{\alpha \lambda
^{1+\frac{2}{\alpha }}+1}{\alpha +2}-\frac{\lambda }{2}\right) .
\label{2-4c}
\end{eqnarray}%
If we use (\ref{2-4b}), (\ref{2-4c}) and (\ref{2-2d}) in (\ref{2-4a}), we
obtain the desired result. This completes the proof.
\end{proof}

\begin{corollary}
Under the assumptions of Theorem \ref{2.3} with $q=1,$ the inequality (\ref%
{2-4}) reduced to the following inequality%
\begin{eqnarray*}
\left\vert S_{f}\left( mx,\lambda ,\alpha ,ma,mb\right) \right\vert  &\leq
&\left\{ \frac{m^{\alpha }\left( x-a\right) ^{\alpha +1}}{b-a}\left(
\left\vert f^{\prime }\left( mx\right) \right\vert A_{2}\left( \alpha
,\lambda \right) +m\left\vert f^{\prime }\left( a\right) \right\vert
A_{3}\left( \alpha ,\lambda \right) \right) \right.  \\
&&\left. +\frac{m^{\alpha }\left( b-x\right) ^{\alpha +1}}{b-a}\left(
\left\vert f^{\prime }\left( mx\right) \right\vert A_{2}\left( \alpha
,\lambda \right) +m\left\vert f^{\prime }\left( b\right) \right\vert
A_{3}\left( \alpha ,\lambda \right) \right) \right\} .
\end{eqnarray*}
\end{corollary}

\begin{corollary}
Under the assumptions of Theorem \ref{2.3} with $x=\frac{a+b}{2},\ \lambda =%
\frac{1}{3},$from the inequality (\ref{2-4}) we get the following Simpson
type inequality or fractional integrals%
\begin{eqnarray*}
&&\left\vert \frac{2^{\alpha -1}}{m^{\alpha }\left( b-a\right) ^{\alpha -1}}%
S_{f}\left( m\left( \frac{a+b}{2}\right) ,\frac{1}{3},\alpha ,ma,mb\right)
\right\vert  \\
&=&\left\vert \frac{1}{6}\left[ f(ma)+4f\left( \frac{m\left( a+b\right) }{2}%
\right) +f(mb)\right] -\frac{\Gamma \left( \alpha +1\right) 2^{\alpha -1}}{%
m^{\alpha }\left( b-a\right) ^{\alpha }}\left[ J_{\left( \frac{m\left(
a+b\right) }{2}\right) ^{-}}^{\alpha }f(ma)+J_{\left( \frac{m\left(
a+b\right) }{2}\right) ^{+}}^{\alpha }f(mb)\right] \right\vert  \\
&\leq &\frac{m\left( b-a\right) }{4}A_{1}^{1-\frac{1}{q}}\left( \alpha ,%
\frac{1}{3}\right) \left\{ \left( \left\vert f^{\prime }\left( \frac{m\left(
a+b\right) }{2}\right) \right\vert ^{q}A_{2}\left( \alpha ,\frac{1}{3}%
\right) +m\left\vert f^{\prime }\left( a\right) \right\vert ^{q}A_{3}\left(
\alpha ,\frac{1}{3}\right) \right) ^{\frac{1}{q}}\right.  \\
&&\left. +\left( \left\vert f^{\prime }\left( \frac{m\left( a+b\right) }{2}%
\right) \right\vert ^{q}A_{2}\left( \alpha ,\frac{1}{3}\right) +m\left\vert
f^{\prime }\left( b\right) \right\vert ^{q}A_{3}\left( \alpha ,\frac{1}{3}%
\right) \right) ^{\frac{1}{q}}\right\} .
\end{eqnarray*}
\end{corollary}

\begin{corollary}
Under the assumptions of Theorem \ref{2.3} with $x=\frac{a+b}{2},\ \lambda
=0,$from the inequality (\ref{2-4}) we get the following midpoint type
inequality or fractional integrals%
\begin{eqnarray*}
&&\frac{2^{\alpha -1}}{m^{\alpha }\left( b-a\right) ^{\alpha -1}}\left\vert
S_{f}\left( m\left( \frac{a+b}{2}\right) ,0,\alpha ,ma,mb\right) \right\vert 
\\
&=&\left\vert f\left( \frac{m\left( a+b\right) }{2}\right) -\frac{\Gamma
\left( \alpha +1\right) 2^{\alpha -1}}{m^{\alpha }\left( b-a\right) ^{\alpha
}}\left[ J_{\left( \frac{m\left( a+b\right) }{2}\right) ^{-}}^{\alpha
}f(ma)+J_{\left( \frac{m\left( a+b\right) }{2}\right) ^{+}}^{\alpha }f(mb)%
\right] \right\vert  \\
&\leq &\frac{m\left( b-a\right) }{4}\left( \frac{1}{\alpha +1}\right) \left( 
\frac{1}{\alpha +2}\right) ^{\frac{1}{q}}\left\{ \left[ \left( \alpha
+1\right) \left\vert f^{\prime }\left( \frac{m\left( a+b\right) }{2}\right)
\right\vert ^{q}+m\left\vert f^{\prime }\left( a\right) \right\vert ^{q}%
\right] ^{\frac{1}{q}}\right.  \\
&&\left. +\left[ \left( \alpha +1\right) \left\vert f^{\prime }\left( \frac{%
m\left( a+b\right) }{2}\right) \right\vert ^{q}+m\left\vert f^{\prime
}\left( b\right) \right\vert ^{q}\right] ^{\frac{1}{q}}\right\} .
\end{eqnarray*}
\end{corollary}

\begin{corollary}
Under the assumptions of Theorem \ref{2.3} with$\ \lambda =1,$from the
inequality (\ref{2-4}) we get the following trapezoid type inequality or
fractional integrals%
\begin{eqnarray*}
&&\left\vert S_{f}\left( mx,1,\alpha ,ma,mb\right) \right\vert  \\
&=&\left\vert \frac{\left( x-a\right) ^{\alpha }f(ma)+\left( b-x\right)
^{\alpha }f(mb)}{b-a}-\frac{\Gamma \left( \alpha +1\right) }{m^{\alpha
+1}\left( b-a\right) }\left[ J_{mx^{-}}^{\alpha }f(ma)+J_{mx^{+}}^{\alpha
}f(mb)\right] \right\vert  \\
&\leq &\left( \frac{\alpha }{\alpha +1}\right) \left( \frac{1}{2\left(
\alpha +2\right) }\right) ^{\frac{1}{q}}\left\{ \frac{\left( x-a\right)
^{\alpha +1}}{b-a}\left[ \left( \alpha +1\right) \left\vert f^{\prime
}\left( mx\right) \right\vert ^{q}+\left( \alpha +3\right) m\left\vert
f^{\prime }\left( a\right) \right\vert \right] ^{\frac{1}{q}}\right.  \\
&&\left. +\frac{\left( b-x\right) ^{\alpha +1}}{b-a}\left[ \left( \alpha
+1\right) \left\vert f^{\prime }\left( mx\right) \right\vert ^{q}+\left(
\alpha +3\right) m\left\vert f^{\prime }\left( b\right) \right\vert \right]
^{\frac{1}{q}}\right\} .
\end{eqnarray*}
\end{corollary}

\begin{corollary}
Let the assumptions of Theorem \ref{2.3} hold. If $\ \left\vert f^{\prime
}(u)\right\vert \leq M$ for all $u\in \left[ ma,mb\right] $ and $\lambda =0,$
then from the inequality (\ref{2-4}) we get the following Ostrowski type
inequality or fractional integrals%
\begin{eqnarray*}
&&\left\vert \left[ \frac{\left( x-a\right) ^{\alpha }+\left( b-x\right)
^{\alpha }}{b-a}\right] f(mx)-\frac{\Gamma \left( \alpha +1\right) }{%
m^{\alpha +1}\left( b-a\right) }\left[ J_{mx^{-}}^{\alpha
}f(ma)+J_{mx^{+}}^{\alpha }f(mb)\right] \right\vert  \\
&\leq &mM\left( \frac{1}{\alpha +1}\right) \left( \frac{\alpha +m+1}{\alpha
+2}\right) ^{\frac{1}{q}}\left[ \frac{\left( x-a\right) ^{\alpha +1}+\left(
b-x\right) ^{\alpha +1}}{b-a}\right] ,
\end{eqnarray*}%
for each $x\in \left[ a,b\right] .$
\end{corollary}

\end{document}